\documentclass[letter,12pt]{article}
\usepackage{amssymb}
\usepackage{amsthm}
\usepackage{amsfonts}
\usepackage{amsmath}
\usepackage{anysize}
\usepackage{mathtools}
\usepackage{bbm}
\usepackage{mathrsfs}
\usepackage{abstract}
\marginsize{2.5cm}{2.5cm}{0.9cm}{1.9cm}
\usepackage{color}

\newtheorem{theorem}{Theorem}[section]
\newtheorem{lemma}[theorem]{Lemma}
\newtheorem{prop}[theorem]{Proposition}
\newtheorem{coro}[theorem]{Corollary}

\theoremstyle{definition}
\newtheorem{definition}[theorem]{Definition}

\newcommand{\Ric}{{\rm Ric}}

\newcommand{\PSH}{{\rm PSH}}

\newcommand{\Id}{\operatorname{Id}}

\usepackage{todonotes}

\usepackage[hidelinks]{hyperref}
\makeatletter
\def\Xint#1{\mathchoice
{\XXint\displaystyle\textstyle{#1}}%
{\XXint\textstyle\scriptstyle{#1}}%
{\XXint\scriptstyle\scriptscriptstyle{#1}}%
{\XXint\scriptscriptstyle\scriptscriptstyle{#1}}%
\!\int}
\def\XXint#1#2#3{{\setbox0=\hbox{$#1{#2#3}{\int}$ }
\vcenter{\hbox{$#2#3$ }}\kern-.6\wd0}}

\def\dashint{\Xint-}
\@ifundefined{fint}
{\newcommand\fint{\dashint}}
\makeatother

\title{A Wess--Zumino--Witten type equation in the space of K\"ahler potentials in terms of Hermitian--Yang--Mills metrics}
\author{Kuang-Ru Wu\thanks{Research partially supported by NSF grant DMS-1764167.}}
\date{}

\begin{document}
\maketitle
\begin{abstract}
    We prove that the solution of a Wess--Zumino--Witten type equation from a domain $D$ in $\mathbb{C}^m$ to the space of K\"ahler potentials can be approximated uniformly by Hermitian--Yang--Mills metrics on certain vector bundles. The key is a new version of Berndtsson's theorem on the positivity of direct image bundles.
\end{abstract}
\section{Introduction}

Let $L$ be a positive line bundle over a compact complex manifold $X$ of dimension $n$, and $h$ a positively curved metric on $L$ with curvature $\omega$. The space of K\"ahler potentials is $\mathcal H_{\omega}=\{\phi\in C^\infty(X,\mathbb{R}): \omega +i\partial\Bar{\partial }\phi>0 \}$, and for a positive integer $k$ we denote by $\mathcal H_k$ the space of inner products on $H^0(X,L^k)$. Starting from a question asked by Yau \cite{Yau87} and the work of Tian \cite{Tia90}, Zelditch \cite{Zel98}, Catlin \cite{Cat99}, and many others, it is well-known that a given K\"ahler potential $\phi\in \mathcal{H}_\omega$ can be approximated by $\phi_k\in \mathcal{H}_\omega$ associated with $\mathcal H_k$ as $k\to \infty$. Furthermore, Mabuchi \cite{Ma87}, Semmes \cite{Sem92}, and Donaldson \cite{Don99} discovered that $\mathcal H_\omega$ carries a Riemannian metric which allows one to talk about geometry, especially geodesics, of $\mathcal H_\omega$. Thanks to Phong--Sturm \cite{PS06}, Berndtsson \cite{Bern13}, and Darvas--Lu--Rubinstein \cite{DLR18}, geodesics in $\mathcal H_\omega$ can be approximated by geodesics in $\mathcal H_k$ as $k\to \infty$. More generally, one may wonder if harmonic maps into $\mathcal H_\omega$ can also be approximated by harmonic maps associated with $\mathcal{H}_k$. A version of this was confirmed by Rubinstein--Zelditch \cite{RZ10} when $X$ is toric, and the maps take values in toric K\"ahler metrics (see also \cite{SZ07} and \cite{SZ10}).      

In this article, we focus on a Wess--Zumino--Witten (WZW) type equation for a map from $D\subset \mathbb{C}^m$ to $\mathcal H_\omega$, and we show that the solution to such an equation can be approximated by Hermitian--Yang--Mills metrics on certain direct image bundles. We will also see how this result recovers some of those mentioned in the first paragraph. 

We first explain how to derive this WZW equation. Recall that the tangent space $T_{\phi}\mathcal H_\omega$ at $\phi\in \mathcal H_\omega$ can be canonically identified with $C^\infty (X,\mathbb{R})$, and following Mabuchi \cite{Ma87}, Semmes \cite{Sem92}, and Donaldson \cite{Don99}, the Mabuchi metric $g_M$ on $\mathcal H_\omega$ is $$g_M(\xi,\eta)=\int_X \xi \eta\omega^n_\phi, \textup{ for }\phi\in \mathcal H_\omega \textup{ and } \xi,\eta \in T_\phi\mathcal H_\omega .$$
Let $D$ be a bounded smooth strongly pseudoconvex domain in $\mathbb{C}^m$. A map $\Phi:D\to \mathcal H_\omega$ will be identified as $\Phi: D\times X \to \mathbb{R}$ with $\Phi(z,\cdot)\in \mathcal H_\omega$ for $z\in D$. A map $\Phi:D\to \mathcal H_\omega$ is said to be harmonic if it is a critical point of the functional $E(\Phi)=\int_D |\Phi_*|^2 dV$ where $dV$ is the Euclidean volume form on $D$, $\Phi_*$ is the differential of $\Phi$, and $|\Phi_*|$ is the Hilbert--Schmidt norm of $\Phi_*$, measured by Mabuchi metric $g_M$ and the Euclidean metric of $D$. A straightforward computation gives the harmonic map equation: 
\begin{equation}\label{har}
\sum^m_{j=1}|\nabla\Phi_{z_j}|^2-2\Phi_{z_j\bar{z}_j}=0,    
\end{equation}
where $\{z_j\}$ are coordinates on $D$ and $|\nabla \Phi_{z_j}(z)|^2$ is computed using the metric $\omega_{\Phi(z)}$. On the other hand, there is a perturbed functional $\mathscr{E}$, whose Euler--Lagrange equation is also of interest. The construction of this perturbed functional is similar to that of \cite[Section 5]{Don99} (see also \cite{Wit83}), who dealt with one dimensional $D$. In order to define $\mathscr{E}$, we first define a three-form $\theta$ on $\mathcal{H}_\omega$: for $\phi\in \mathcal H_\omega$ and $ \xi_1,\xi_2,\xi_3 \in T_\phi\mathcal H_\omega$, 
\begin{equation}\label{three form}
  \theta(\xi_1,\xi_2,\xi_3):=g_M(\{\xi_1,\xi_2\}_{\omega_\phi},\xi_3)=\int_X\{\xi_1,\xi_2\}_{\omega_\phi}\xi_3 \omega_\phi^n,  
\end{equation}
where $\{\cdot ,\cdot \}_{\omega_\phi}$ is the Poisson bracket determined by the symplectic form $\omega_\phi$. This three-form $\theta$ is $d$-closed (see Lemma \ref{theta is closed} below), and therefore there is a two-form $\alpha$ on $\mathcal H_\omega$ such that $d\alpha=\theta$. For a map $\Phi: D\to \mathcal{H}_\omega$, we define $\mathscr{E}(\Phi):=E(\Phi)+4 i \sum_j\int_D\alpha(\Phi_{\bar{z}_j},\Phi_{{z_j}})dV$. We will show in Lemma \ref{EL for WZW}, the Euler--Lagrange equation of $\mathscr{E}$ is
\begin{equation}\label{WZW}
  \sum^m_{j=1}|\nabla\Phi_{z_j}|^2-2\Phi_{z_j\bar{z}_j}+i\{\Phi_{\bar{z}_j},\Phi_{{z_j}}\}_{\omega_\Phi}=0.  
\end{equation} 
In view of its connection with \cite{Wit83} and following \cite{Don99}, we call the equation (\ref{WZW}) WZW equation for a map $\Phi: D\to \mathcal{H}_\omega$.

Donaldson showed in \cite{Don99}, when $m=1$, that the WZW equation is equivalent to a homogeneous complex Monge--Amp\`ere equation. We have the following extended equivalence for $m\geq 1$ by a similar computation. Let $\pi:D\times X \to X$ be the projection onto $X$. Then the extended equivalence is 
\begin{equation}\label{HCMA WZW}
    \Phi \textup{ solves (\ref{WZW})} \textup{ if and only if }  (i\partial \bar{\partial}\Phi+\pi^*\omega)^{n+1}\wedge(i\sum_{j=1}^m dz_j\wedge d\bar{z}_j)^{m-1}=0.
\end{equation}
This suggests that the proper generality of the WZW equation is for maps from a K\"ahler manifold $D$ to $\mathcal H_\omega$. Nevertheless, in this paper we restrict to $D\subset \mathbb{C}^m$. 

The next step is to approximate the solution of the WZW equation, and we will use interpolation of Perron families to construct approximants. We first introduce the following definition
\begin{definition}
 We will say that a function $u:D\times X\to [-\infty, \infty)$ is $\omega$-subharmonic on graphs if, for any holomorphic map $f$ from an open subset of $D$ to $X$, $\psi(f(z))+u(z,f(z))$ is subharmonic, where $\psi$ is a local potential of $\omega$.
\end{definition}
This definition does not depend on the choice of $\psi$ since any two local potentials differ by a pluriharmonic function. (This definition has its origin in the works of Slodkowski \cite{Slo88},\cite{Slo90a},\cite{Slo90b}, and Coifman and Semmes \cite{CS93}; however, they focus on functions with values norms or quasi-norms, whereas we consider general functions. There is also a notion of $k$-subharmonicity, see \cite{Blo05}, but it is not equivalent to subharmonicity on graphs.)

Let $v$ be a real-valued smooth function on $\partial D\times X$ and $\partial D \ni z \mapsto v(z,\cdot)=v_z\in \mathcal{H}_{\omega}$. We simply write $v \in C^{\infty}(\partial D,\mathcal{H}_{\omega})$. Consider the Perron family 
\begin{align*}
G_v:=\{u \in\textup{usc($D\times X$)}:& \textup{ $u$ is $\omega$-subharmonic on graphs}, \textup{and }\limsup_{D \ni z \to \zeta\in \partial D}u(z,x)  \leq v(\zeta,x)\}.
\end{align*}
As we will later see, the upper envelope $V=\sup\{u: u \in G_v\}$ is a weak solution of the WZW equation from $D$ to $\mathcal H_\omega$. The above set up is for $\mathcal{H}_\omega$. As for $\mathcal{H}_k$, we recall first the two maps that connect $\mathcal H_\omega$ and $\mathcal H_k$. The Hilbert map $H_k: \mathcal H_{\omega}\to \mathcal H_k$ is $$H_k(\phi)(s,s)=\int_X h^k(s,s)e^{-k\phi}\omega^n, \textup{ for $\phi\in \mathcal{H}_\omega$ and $s\in H^0(X,L^k)$}.$$ In the other direction, the Fubini--Study map $FS_k: \mathcal H_k\to \mathcal H_{\omega}$ is
\begin{equation*}
FS_k(G)(x)=\frac{1}{k}\log \sup_{s\in H^0(X,L^k),G(s,s)\leq 1}h^k(s,s)(x), \textup{ for $G\in \mathcal H_k$ and } x\in X.
\end{equation*}
Now following the definitions from \cite{CS93}, let $\mathcal N_k^*$ be the set of norms on $H^0(X,L^k)^*$, then a norm function $D\ni z\mapsto U_z\in \mathcal N_k^*$ is said to be subharmonic if $\log U_z(f(z))$ is subharmonic for any holomorphic section $f:W\subset D\to H^0(X,L^k)^*$. The second Perron family we consider is  
\begin{align*}
  G^{k}_v:=\{&D \ni z \to U_z \in \mathcal N_k^* \textup{ is subharmonic and }\\ 
  &\limsup_{D\ni z \to \zeta \in \partial D} U^2_z(s) \leq H_k^*(v_\zeta)(s,s) \textup{ for any $s\in H^0(X,L^k)^*$ }  \},  
\end{align*}
where $H_k^*(v)$ is the inner product dual to $H_k(v)$. A remarkable fact about the upper envelope $V^{k}=\sup \{U:U\in G^{k}_v\}$ is a theorem of Coifman and Semmes \cite{CS93}, which shows that $V^{k}$ is not only a norm but an inner product (see \cite[Corollary 2.7]{Slo90a} for a different proof.); moreover it solves the Hermitian--Yang--Mills equation:
\begin{equation*}
\begin{cases}
\Lambda \Theta(V^{k})=0\\
V^{k}|_{\partial D}=H^*_k(v).
\end{cases}    
\end{equation*} 
Here we view $V^{k}$ as a Hermitian metric on the bundle $\overline{D}\times H^0(X,L^k)^*\to \overline{D}$, and $\Theta(V^{k})$ is its curvature, while $\Lambda$ is the trace with respect to the Euclidean metric of $D$, so in general $\Lambda \Theta(V^{k})$ takes values in endomorphisms of $H^0(X,L^k)^*$. Denoting the dual metric by $(V^k)^*$, our main result is that the upper envelope $V$ of $G_v$ is the limit of Hermitian--Yang--Mills metrics:
\begin{theorem}\label{thm:YM approx}
$FS_k((V^{k})^*)$ converges to $V$ uniformly on $D \times X$, as $k\to \infty$.
\end{theorem}

Now we turn to the interpretation of the upper envelope $V$ and its relation with the WZW equation. 
The next theorem shows that $V$ solves the WZW equation under a regularity assumption.
\begin{theorem}\label{thm har}
If the upper envelope $V$ of $G_v$ is in $C^2(D\times X)$, then $$(i\partial \bar{\partial}V+\pi^*\omega)^{n+1}\wedge(i\sum_{j=1}^m dz_j\wedge d\bar{z}_j)^{m-1}=0.$$ 
\end{theorem}
As a result, Theorems \ref{thm:YM approx} and \ref{thm har} together show that the solution of the WZW equation can be approximated by the Hermitian--Yang--Mills metrics. (The equation in Theorem \ref{thm har} is similar to the complex Hessian equation, which has been studied extensively e.g. \cite{Blo05}, \cite{DK14}, \cite{CP19}, and \cite{LN19}, and we hope to return to it in the future).

We mention briefly works related to our result. If $m=1$ and $D\subset \mathbb{C}$ is an annulus, and $v$ is invariant under rotation of the annulus, then Theorems \ref{thm:YM approx} and \ref{thm har} recover the geodesic approximation result of Phong--Sturm \cite{PS06} and Berndtsson \cite{Bern13}. When $X$ is toric, these theorems are reduced to the harmonic approximation of Rubinstein--Zelditch \cite{RZ10}, except that $C^2$ convergence is proved in their paper (see also \cite{SZ07} and \cite{SZ10}). 

The proof of Theorem \ref{thm:YM approx} hinges on Theorem \ref{thm bo bundle}, a result regarding the positivity of direct image bundles. Although Berndtsson's theorem \cite{Bern09} has played a crucial role in approximation theorems similar to Theorem \ref{thm:YM approx} (for example \cite{Bern13}, \cite{BK12}, \cite{DLR18}, and \cite{DW19}), when it comes to approximating by Hermitian--Yang--Mills metrics, a subharmonic analogue of Berndtsson's theorem is desired. It is Theorem \ref{thm bo bundle}, where we prove a version of positivity of direct image bundles for weights that are subharmonic on graphs. This is perhaps the crux of this paper. A corresponding result on Stein manifolds can be proved easily following the proof of Theorem \ref{thm bo bundle}. 

 The WZW equation (\ref{WZW}) is the harmonic map equation (\ref{har}) perturbed with Poisson's bracket, which is closely related to the geometry of $\mathcal{H}_\omega$, an infinite dimensional nonpositively curved manifold. Since the theory of harmonic maps into nonpositively curved manifolds is well developed by Eells--Sampson \cite{ES64}, Hamilton \cite{Ham75}, and many others, a possible future direction is to see if one can combine the classical results with those of this paper to study $\mathcal{H}_\omega$.
  
Before we end this introduction, a few words about the structure of this paper. In section \ref{ sec subharmonic bo}, the subharmonic version of positivity of direct image bundles is proved, except we put off a technical lemma to section \ref{section approx}. Section \ref{sec quan} is devoted to Theorem \ref{thm:YM approx}, and section \ref{sec har} to Theorem \ref{thm har}. In the final section \ref{compare}, we draw parallels with the paper \cite{DW19}.

This paper grew out of joint work \cite{DW19} with Tam\'as Darvas, and I am grateful to him for many useful discussions. I am indebted to L\'aszl\'o Lempert for his critical remarks and suggestions. I would like to thank Chi Li and Jiyuan Han for stimulating conversations.

\section{Positivity of direct image bundles}\label{ sec subharmonic bo}
 
Consider a Hermitian holomorphic line bundle $(E,g)\to X^n$ over a compact complex manifold and assume the curvature $\eta$ of the metric $g$ is positive. We define a variant of the Hilbert map: $\text{Hilb}_{E\otimes K_X}(u)$, for a function $u: D\times X\to \mathbb{R}$, is given by  $$\text{Hilb}_{E\otimes K_X}(u)(s,s)=\int_X g(s,s)e^{-u(z,\cdot)}$$ with $s\in H^0(X,E\otimes K_X)$. In the following, suitable assumptions will be made on $u$ to make sure the integral converges. Then the map $z\mapsto \text{Hilb}_{E\otimes K_X}(u)$ is a Hermitian metric on the bundle $D\times H^0(X,E\otimes K_X)\to D$.
The main result of this section is the following positivity theorem.
\begin{theorem}\label{thm bo bundle} 
If $u$ is bounded and upper semicontinuous (usc) on $D\times X$, and $\eta$-subharmonic on graphs, then the dual metric $\text{Hilb}^*_{E\otimes K_X}(u)$ is a subharmonic norm function.
\end{theorem}

The following approximation lemma is somewhat technical and we postpone its proof to section \ref{section approx}.
\begin{lemma}\label{approx}
 Let $u$ be a bounded usc function on $D\times X$, $\eta$-subharmonic on graphs. Then for $D'$ relatively compact open in $D$, there exist $\varepsilon_j \searrow 0$ and $u_j\in C^\infty(D'\times X)$ decreasing to $u$ such that for any holomorphic map $f$ from an open subset of $D'$ to $X$, $\Delta(\psi(f(z))+u_j(z,f(z)))\geq \varepsilon_j\Delta(\psi(f(z))$, where $\eta=i\partial \Bar{\partial}\psi $ locally.
\end{lemma}

\begin{proof}[Proof of Theorem \ref{thm bo bundle}]
Since being a subharmonic norm function is a local property, we focus on $D'$, a relatively compact open set in $D$. Take $\varepsilon_j$ and $u_j$ as in Lemma \ref{approx}. Assuming the theorem holds for such a $u_j$, namely, the dual metric $\text{Hilb}^*_{E\otimes K_X}(u_j)$ is a subharmonic norm function, it follows that $\text{Hilb}^*_{E\otimes K_X}(u)$ is also a subharmonic norm function because $\text{Hilb}^*_{E\otimes K_X}(u_j)$ decreases to $\text{Hilb}^*_{E\otimes K_X}(u)$ as $j\to \infty$.

As a result, we only need to prove the theorem for $u\in C^\infty(D'\times X)$ with the property that there exists $\varepsilon>0$ such that for any holomorphic function $f$ from an open subset of $D'$ to $X$, 
\begin{equation}\label{epsilon}
 \Delta(\psi(f(z))+u(z,f(z)))\geq \varepsilon\Delta(\psi(f(z)) \textup{, where $\eta=i\partial \Bar{\partial}\psi$ locally.}    
\end{equation}

In a coordinate system $\Omega\subset \mathbb{C}^n$ on $X$, we will use Greek letters $\mu,\lambda$ for indices of coordinates on $X$, and Roman letters $i,j$ for indices of coordinates on $D$; moreover, $f^\mu$ means the $\mu$-th component of $f$, whereas $\psi_{\mu\Bar{\lambda}}$, $u_{i\Bar{i}}$, and $u_{i\Bar{\lambda}}$ mean partial derivatives $\partial^2\psi/\partial x_\mu \partial\bar{x}_\lambda$, $\partial^2 u/\partial z_i \partial\bar{ z_i}$, and $\partial^2 u/\partial z_i \partial\bar{x}_\lambda$ respectively. In this coordinate system $\Omega\subset \mathbb{C}^n$ on $X$, the inequality (\ref{epsilon}) becomes
\begin{align}\label{subhar ineq}
\varepsilon\sum_{i, \lambda, \mu}\psi_{\mu \bar{\lambda} }\frac{\partial f^{\mu}}{\partial z_i}\frac{\partial \Bar{f^{\lambda}}}{\partial \Bar{z_i}}\leq
\begin{multlined}[t]
\sum_{i, \lambda, \mu}\psi_{\mu \bar{\lambda} }\frac{\partial f^{\mu}}{\partial z_i}\frac{\partial \Bar{f^{\lambda}}}{\partial \Bar{z_i}}+\sum_iu_{i\bar{i}}\\
+\sum_{i,\lambda }u_{i \bar{\lambda}}\frac{\partial \Bar{f^{\lambda}}}{\partial \Bar{z_i}}+\sum_{i,\mu}u_{\bar{i}\mu}\frac{\partial f^{\mu}}{\partial z_i}+\sum_{i, \lambda, \mu}u_{\mu \bar{\lambda} }\frac{\partial f^{\mu}}{\partial z_i}\frac{\partial \Bar{f^{\lambda}}}{\partial \Bar{z_i}}. 
\end{multlined}
\end{align}
In (\ref{subhar ineq}), 
choose $f(z)=N(\xi_1,\xi_2,...,\xi_n)z_1$ where $N$ is a positive number and $(\xi_1,\xi_2,...,\xi_n)\in \mathbb{C}^n$, divide the resulting (\ref{subhar ineq}) by $N^2$ and send $N$ to infinity, to obtain $(\psi_{\mu\Bar{\lambda}}+u_{\mu\Bar{\lambda}})\geq \varepsilon(\psi_{\mu\Bar{\lambda}})$ as matrices, and hence $(\psi_{\mu\Bar{\lambda}}+u_{\mu\Bar{\lambda}})$ is positive definite.

Let $L^2(X,E\otimes K_X)$ be the space of measurable sections $s$ whose $L^2$ norm $\int_Xg(s,s)e^{-u(z,\cdot)}$ is finite. Since different $z$ will give rise to comparable $L^2$ norms, the space $L^2(X,E\otimes K_X)$ does not change with $z$, and so we have a Hermitian Hilbert bundle $D'\times L^2(X,E\otimes K_X)\to D'$ which has $D'\times H^0(X,E\otimes K_X)\to D'$ as a subbundle. Denote the curvature of the subbundle by $\Theta=\sum \Theta_{j\bar{k}}dz_j\wedge dz_{\Bar{k}}$. Following the computations in the proof of \cite[Theorem 1.1]{Bern09} we deduce 
\begin{equation}\label{L2} 
\sum_{j}(\Theta_{j\Bar{j}}s,s)\geq \int_{X} K(z,\cdot)  g(s,s)e^{-u(z,\cdot)}
\end{equation}
where $s\in H^0(X,E\otimes K_X)$, and $K:D'\times X \to \mathbb{R}$ is a smooth function, given in local coordinates on $X$ by $$K=\sum_j(u_{j\Bar{j}}-\sum_{\lambda, \mu}(\psi+u)^{\Bar{\lambda} \mu}u_{j\Bar{\lambda}}u_{\Bar{j}\mu});$$ here $(\psi+u)^{\Bar{\lambda} \mu}$ stands for the inverse matrix of $(\psi+u)_{\Bar{\lambda} \mu}$. 

We claim that $K\geq 0$. First notice that $\psi$ is independent of $z$, so if we denote $\psi(x)+u(z,x)$ by $\phi(z,x)$, then $K=\sum_j(\phi_{j\Bar{j}}-\sum_{\lambda, \mu}\phi_{j\Bar{\lambda}}\phi^{\Bar{\lambda} \mu}\phi_{\Bar{j}\mu})$. Fix $(z_0,x_0)\in D'\times X$, since the matrix $(\phi_{\mu\bar{\lambda}})$ is positive definite, we can choose local coordinates on $X$ around $x_0$ such that $(\phi_{\mu\bar{\lambda}})$ is the identity matrix at $(z_0,x_0)$, and therefore $K(z_0,x_0)=\sum_j(\phi_{j\Bar{j}}-\sum_\lambda |\phi_{j\Bar{\lambda}}|^2)(z_0,x_0)$. For a holomorphic function $f$ from an open subset of $D'$ to $X$, the subharmonicity of $\phi(z,f(z))$ reads  
\begin{equation}\label{sh}
    \sum_i\phi_{i\Bar{i}}+\sum_{i ,\lambda}\phi_{i \Bar{\lambda}}\frac{\partial \Bar{f^{\lambda}}}{\partial \Bar{z_i}}+\sum_{i,\mu}\phi_{\Bar{i}\mu}\frac{\partial f^{\mu}}{\partial z_i}+\sum_{i ,\lambda, \mu}\phi_{ \mu\Bar{\lambda}}\frac{\partial f^{\mu}}{\partial z_i}\frac{\partial \Bar{f^{\lambda}}}{\partial \Bar{z_i}}\geq 0.
\end{equation}
Without loss of generality, we assume $(z_0,x_0)=(0,0)$ and choose $f^{\lambda}=-\sum_i\phi_{i\Bar{\lambda}}(0,0) z_i$ in (\ref{sh}), and it becomes $\sum_j(\phi_{j\Bar{j}}-\sum_\lambda |\phi_{j\Bar{\lambda}}|^2)(0,0) \geq 0$. Therefore, $K\geq 0$. (See also the remark after Lemma \ref{+ det} for a slightly different proof of this claim, and an invariant meaning of $K$).

As a result, (\ref{L2}) implies  $\sum_j(\Theta_{j\Bar{j}}s,s)\geq 0$, and hence the curvature of the dual metric $\text{Hilb}^*_{E\otimes K_X}(u)$ satisfies the opposite inequality; according to \cite[Theorem 4.1]{CS93}, this implies $\text{Hilb}^*_{E\otimes K_X}(u)$ is a subharmonic norm function. 
\end{proof}

Now we replace $(E,g)$ by $(L^k\otimes K_X^*,h^k\otimes \omega^n)$, which is positively curved for large $k$ since $\Theta(h^k\otimes \omega^n)=k\omega+\Ric \, \omega$. We have the following proposition regarding the metric $H_k(u)$ on the bundle $D\times H^0(X,L^k)$.
\begin{prop}\label{prop}
    Suppose $u$ is a bounded usc function on $D\times X$ and with some $\varepsilon\in (0,1)$ $u$ is $(1-\varepsilon)\omega$-subharmonic on graphs. Then there exists $k_0=k_0(\varepsilon,\omega)$, independent of $u$, such that, for $k\geq k_0$, the dual metric $H^*_k(u)$ is a subharmonic norm function.
\end{prop}
\begin{proof}
   In order to use Theorem \ref{thm bo bundle}, we check if $ku$ is $(k\omega+\Ric\, \omega)$-subharmonic on graphs. Suppose $\omega=i\partial\Bar{\partial}\psi$ and $\Ric\, \omega=i\partial\Bar{\partial}\phi$ locally, then we want to see if $k\psi(f(z))+\phi(f(z))+ku(z,f(z))$ is subharmonic for any holomorphic map $f$. Note that $k\psi+\phi+ku = k(1-\varepsilon)\psi+ku+\varepsilon k\psi+\phi$, and $k(1-\varepsilon)\psi(f(z))+ku(z,f(z))$ is subharmonic by the assumption. On the other hand, there exists $k_0$ depending on $\varepsilon, \omega$ such that $\varepsilon k\psi+\phi$ is plurisubharmonic (psh) for $k\geq k_0$. Therefore, $ku$ is $(k\omega+\Ric\, \omega)$-subharmonic on graphs for $k\geq k_0$. By Theorem \ref{thm bo bundle}, the metric $\text{Hilb}^*_{L^k}(ku)$ is a subharmonic norm function for $k\geq k_0$. The proposition follows since $\text{Hilb}_{L^k}(ku)=H_k(u)$. 
   
\end{proof}    

\section{Approximation by Hermitian--Yang--Mills metrics} \label{sec quan}

 Recall that $D$ is in $\mathbb{C}^m$, and $(L,h)\to X^n$ is a positive line bundle with curvature $\omega$.

\begin{lemma}\label{lemma psh in X}
Let $u$ be an usc function on $D\times X$, $\omega$-subharmonic on graphs. Then for any fixed $z\in D$, $u(z,x)$ is $\omega$-psh on $X$, and for any fixed $x\in X$, $u(z,x)$ is subharmonic on $D$.
\end{lemma}
This can be seen as a special case of an abstract theorem in \cite[Section 1]{Slo90a}, whose proof we translate to our setting.
\begin{proof}
By choosing the holomorphic map $f$ constant in the definition of $\omega$-subharmonic on graphs, it follows immediately that $u(z,x)$ is subharmonic in $z.$

For a fixed $z_0\in D$, we want to show $x\mapsto \psi(x)+u(z_0,x)$ is psh in a coordinate system on $X$, where $\psi$ is a local potential of $\omega$. Without loss of generality, it suffices to prove that $\mathbb{C}\ni \lambda \mapsto \psi(\lambda e_1)+u(0,\lambda e_1)$ is subharmonic, where $e_1=(1,0,...,0)\in \mathbb{C}^n$. Let $U=\{\lambda\in \mathbb{C}:|\lambda-a|<R\}$ and $h(\lambda)$ harmonic on $U$ and continuous up to boundary. We will be done if $$\psi(a e_1)+u(0,a e_1)+h(a) \leq \max_{\lambda\in \partial U} \psi(\lambda e_1)+u(0,\lambda e_1)+h(\lambda).$$
Suppose the inequality is not true. By \cite[Lemma 4.5]{Slo86}, there is an $\mathbb{R}$-linear function $l:\mathbb{C}\to\mathbb{R}$ and $b\in U$ such that, if we denote 
\begin{equation}\label{summation of stuff}
v(z,\lambda)=\psi(\lambda e_1)+u(z,\lambda e_1)+h(\lambda)+l(\lambda)    
\end{equation}
then  $$v(0,b)>v(0,\lambda), \textup{ for }\lambda\in U-\{b\}.$$
Now define $W(z,\lambda_1,...,\lambda_m):=v(z,\lambda_1)+...+v(z,\lambda_m)$ in a neighborhood of $(0,b^*):=(0,b,...,b)$ in $\mathbb{C}^m\times \mathbb{C}^m$. As $W(0,b^*)>W(0,\lambda_1,...,\lambda_m)$ for $(\lambda_1,...,\lambda_m)\neq b^*$, there exists a ball $B \subset \mathbb{C}^m$ of radius $r$ centered at $b^*$ such that $$W(0,b^*)> \max_{\{0\}\times \partial B } W.$$
Since $W$ is usc, there exists $\varepsilon>0$ such that $W(z, \lambda_1,...,\lambda_m)<W(0,b^*)$, for $|z|\leq \varepsilon$ and $(\lambda_1,...,\lambda_m)\in \partial B$. Let $S=r/\varepsilon \Id_{\mathbb{C}^m}$. We have $W(z,b^*+S(z))<W(0,b^*)$ for $|z|=\varepsilon$, which contradicts the maximum principle because $W(z,b^*+S(z))= \sum^m_{i=1} v(z,b+r/\varepsilon z_i)$ is subharmonic by (\ref{summation of stuff}). 


\end{proof}

Although in the introduction the boundary data $v$ is in $C^\infty(\partial D, \mathcal H_\omega)$, we will prove a lemma for a broader class of boundary data $\nu$. Let $\nu$ be a continuous map $\partial D\times X \to \mathbb{R}$ such that $\nu_z(\cdot):=\nu(z,\cdot)\in \PSH(X,\omega)$ for $z\in \partial D$.  Let 
\begin{align*}
G_\nu=\{u \in\textup{usc($D\times X$)}:& \textup{ $u$ is $\omega$-subharmonic on graphs}, \textup{and }\limsup_{D \ni z \to \zeta\in \partial D}u(z,x)  \leq \nu(\zeta,x)\}.
\end{align*}
In order to study the properties of the upper envelope $\mathcal V$ of $G_\nu$, we introduce a closely related family. With $\pi:D\times X \to X$ the projection, let $$F_\nu:=\{u:u\in \text{PSH}(D\times X,  \pi^*\omega)\textup{ and }\limsup_{D \ni z \to \zeta\in \partial D}u(z,x)  \leq \nu(\zeta,x)\}.$$ The upper envelope of $F_\nu$ extends to a solution $\mathcal U\in C(\overline{D}\times X)$ of 
\begin{equation*}
    \begin{cases}
    (\pi^*\omega +i\partial \bar{\partial}\mathcal U)^{n+m}=0 \text{  on $D\times X$}\\
    \pi^*\omega +i\partial \bar{\partial}\mathcal U\geq 0 \text{  on $D\times X$}\\
    \mathcal U|_{\partial D\times X}=\nu,
    \end{cases}
\end{equation*}
see for example \cite{Bou12, DW19}. In addition, we also need the solution $h$ to the Dirichlet problem 
\begin{equation*}
\begin{cases}
\sum_j h_{j\Bar{j}}+\Delta_\omega h+2n=0 \text{  on $D\times X$}\\
h|_{\partial D\times X}=\nu.
\end{cases}    
\end{equation*} 

\begin{lemma}\label{envelope}
If we denote the upper envelopes of $G_\nu$ and $F_\nu$ by $\mathcal V$ and $\mathcal U$ respectively, then  $\mathcal U\leq \mathcal V\leq h$ and $\lim_{(z,x)\to (z_0,x_0)\in \partial D\times X}\mathcal V(z,x)= \nu(z_0,x_0)$. Moreover, if $\nu$ is negative, then so is $\mathcal V$.
\end{lemma}
\begin{proof}
   Unraveling the definitions of $F_\nu$ and $G_\nu$, we see $F_\nu\subset G_\nu$, so $\mathcal U\leq \mathcal V$. For any $u\in G_\nu$, by Lemma \ref{lemma psh in X}, $u(z,\cdot)$ is $\omega$-psh for fixed $z$, hence $\Delta_\omega u+2n\geq 0$; in addition, $u(\cdot,x)$ is subharmonic for fixed $x$. By the maximum principle, $u\leq h$ and hence $\mathcal V\leq h$. $\mathcal U$ and $h$ are both equal to $\nu$ on $\partial D\times X$, and so is $\mathcal V$.
   
   For a fixed $x_0\in X$, let $H_0(z)$ be the harmonic function on $D$ with boundary values $\nu(z,x_0)$. For $u\in G_\nu$, we have $u(z,x_0)\leq H_0(z)$, and therefore $\mathcal V(z,x_0)\leq H_0(z)$. The second statement follows at once.
\end{proof}

With Proposition \ref{prop} at hand, we can start to prove Theorem \ref{thm:YM approx}. The following envelope will be used in the proof: for an usc function $F$ on $X$, we introduce $P(F):=\sup\{h \in \textup{PSH}(X,\omega) \textup{ such that } h \leq  F\} \in \textup{PSH}(X,\omega)$ (see \cite{Berm19}).
\begin{proof}[Proof of Theorem \ref{thm:YM approx}]
   Without loss of generality, we assume $v\leq 0$. Fix $\delta >1$, and for $z\in \partial D$, define $v^\delta_z=P(\delta v_z)$. By \cite[Lemma 4.9]{DW19}, $\partial D\times X \ni (z,x)\mapsto v^\delta_z(x)$ is continuous. Let $V^{\delta}$ be the upper envelope of $G_{v^\delta}$. By Lemma \ref{envelope}, $V^\delta \leq 0$, and so $u \leq 0$ for $u \in G_{v^\delta}$.  The next step is to have a better upper bound for $u \in G_{v^\delta}$. To that end, we can look instead at $\max\{u, c\}$, which is still in $G_{v^\delta}$ as long as the constant $c\leq \min v^\delta$. Since $\max\{u, c\}$ is bounded, we will assume $u$ is bounded. Moreover, $u/\delta$ is $\omega/\delta$-subharmonic on graphs. According to Proposition \ref{prop}, there exists $k_0=k_0(\delta)$ such that for $k\geq k_0$, $H_k^*(u/\delta)$ is a subharmonic norm function. Because $\limsup_{\partial D} H_k^*(u/\delta)\leq  H^*_k(v)$, it follows that $H_k^*(u/\delta)\in G^{k}_v$ and therefore $H_k^*(u/\delta)\leq V^{k}$ on $D$ and $FS_k(H_k(u/\delta))\leq FS_k((V^{k})^*)$. By Lemma \ref{lemma psh in X}, we have $\omega +i\partial \Bar{\partial}u/\delta\geq (1-1/\delta)\omega$, ($\partial \Bar{\partial}$ on $X$). The Ohsawa--Takegoshi extension theorem implies (see \cite[Theorem 2.11]{DLR18} or \cite[Lemma 4.10]{DW19}) that there exist $C>0$ and $k_0(\delta)$ such that, for $k\geq k_0$,  $$\frac{1}{\delta} u- \frac{C}{k}  \leq FS_k \circ H_k\Big(\frac{1}{\delta} u\Big)\leq FS_k((V^k)^*) .$$
   Since $\delta v\leq0$, both $V^{\delta}$ and $u$ are negative by Lemma \ref{envelope}, and as a result we have $u-C/k\leq FS_k((V^{k})^*)$; this is true for any $u \in G_{v^\delta }$, so we actually have $V^{\delta}-C/k\leq FS_k((V^{k})^*)$. In addition, since $v_z+(\delta-1)\inf_{\partial D\times X}(v_z)$ is a competitor in $P(\delta v_z)$, $$V+(\delta-1)\inf_{\partial D\times X}(v)\leq V^{\delta}.$$ Putting things together, we conclude
   \begin{equation}\label{lower}
       V+(\delta-1)\inf_{\partial D\times X}(v)-\frac{C}{k}\leq FS_k((V^{k})^*), \textup{ for } k\geq k_0(\delta).
   \end{equation}
   
   Next we claim that $FS_k ((V^{k}_z)^*)(x)$ is $\omega$-subharmonic on graphs. Some preparation is needed. Let $s$ be a non-vanishing holomorphic section of $L^k$ over an open set $Y\subset X$. Let $e^{-k\phi}:=h^k(s,s)$ and $s^*_k:Y \to (L^k)^*$ be defined by $s^*_k(x)(\cdot)=h^k(\cdot,e^{k\phi(x)/2}s(x))$ for $x\in Y$. Suppose $\hat s^*_k: Y\to H^0(X,L^k)^*$ is the pointwise evaluation map of $s^*_k$, namely $\hat s^*_k(x)(\sigma):=s^*_k(x)(\sigma(x))$ for $\sigma\in H^0(X,L^k)$. Then we have the following formula, which is taken from \cite[Lemma 4.1]{DW19},
\begin{equation}
 FS_k ((V^{k}_z)^*)(x)= \frac{2}{k} \log\big[ V^{k}_z (\hat{s}^*_k(x))\big], \ x \in Y.
\end{equation}
Meanwhile, for $\sigma\in H^0(X,L^k)$, $e^{k\phi(x)/2}\hat s^*_k(x) (\sigma)=\sigma(x)/s(x)$ 
is holomorphic, so $e^{k\phi/2}\hat s^*_k$ is holomorphic. Hence for any holomorphic map $g$ from an open subset of $D$ to $X$
\begin{equation}\label{3}
  \Delta\big(\phi(g(z))+FS_k ((V^{k}_z)^*)(g(z))\big)=\Delta(\frac{1}{k} \log\big[ V^{k}_z ((e^{k\phi/2}\hat s^*_k)\circ g(z)) \big]^2).  
\end{equation}
By \cite[Theorem 4.1]{CS93} the Hermitian--Yang--Mills metric $V^{k}_z$ is a subharmonic norm function, so the last term of (\ref{3}) is nonnegative, which means $FS_k ((V^{k})^*)$ is $\omega$-subharmonic on graphs as we claimed. Further, according to the Tian--Catlin--Zelditch asymptotic theorem or by \cite[Lemma 4.10]{DW19} an easier but cruder estimate, $FS_k ((V^{k}_z)^*|_{\partial D})=FS_k (H_k(v))\leq v+O(\log k/k)$, so $FS_k ((V^{k})^*)\in G_{v+O(\log k /k)}$ and $FS_k ((V^{k})^*) \leq V+O(\log k/k)$. This last inequality together with (\ref{lower}) concludes the proof.
\end{proof}
It is natural to ask if $V$ belongs to $G_v$. A standard approach to show the envelope belongs to a family is to take upper regularization, and the case at hand is very similar to \cite[Lemma 11.11]{CS93}, where upper regularization is taken in the $z$-variables. The reason it works in their lemma  is because their function in the $x$-variables is a norm, but ours is not and regularization does not seem to work. However, with Theorem \ref{thm:YM approx} one can easily show $V\in G_v$. It would be interesting to prove $V\in G_v$ directly without using Theorem \ref{thm:YM approx}, after all $G_v$ and $V$ can be defined on any K\"ahler manifold $(X,\omega)$ without reference to a line bundle.
\begin{coro}\label{cor stable}
The upper envelope $V$ is continuous and $V\in G_v$.
\end{coro}
\begin{proof}
   The first statment is a direct consequence of Theorem \ref{thm:YM approx}. As to the second statement, let $\psi$ be a local potential of $\omega$ and $f$ a holomorphic map from an open subset of $D$ to $X$. For any $u\in G_v$, $\psi(f(z))+u(z,f(z))$ is subharmonic; hence $\psi(f(z))+V(z,f(z))$, the supremum over $u\in G_v$,  is also subharmonic since $V$ is continuous. By Lemma \ref{envelope}, it follows $V\in G_v$.
\end{proof}


\section{The WZW equation}\label{sec har}

We will prove Theorem \ref{thm har} and compute the Euler--Lagrange equation of $\mathscr{E}$ in this section. We begin with an observation. Suppose $u$ is a $C^2$ function on $D\times X$ and $\psi$ is a local potential of $\omega$. Consider the complex Hessian of $u+\psi$ with respect to a fixed coordinate $z_j$ in $D$ and local coordinates $x$ in $X$ where $\psi$ is defined
\begin{equation}\label{matrix}
\left (
\begin{array}{cccc}
(u+\psi)_{z_j\bar{z}_j}
& (u+\psi)_{z_j\bar{x}_1} & \cdots &  (u+\psi)_{z_j\Bar{x}_n}\\
(u+\psi)_{x_1\bar{z}_j} & (u+\psi)_{x_1\bar{x}_1} & \cdots & (u+\psi)_{x_1\bar{x}_n}\\
 \vdots & \vdots & \ddots & \vdots \\
(u+\psi)_{x_n\bar{z}_j}& (u+\psi)_{x_n\bar{x}_1} &\cdots & (u+\psi)_{x_n\bar{x}_n} 
\end{array}
\right ),
\end{equation}
which we will denote by $(u+\psi)_j$. Then
\begin{equation}\label{invariant}
\begin{aligned}
  &(i\partial \bar{\partial}u+\pi^*\omega)^{n+1}\wedge(i\sum_{j=1}^m dz_j\wedge d\bar{z}_j)^{m-1}\\
  =&(n+1)! (m-1)!\sum^m_{j=1}\det (u+\psi)_j  \big(\bigwedge^m_{k=1} i dz_k\wedge d\Bar{z}_k\wedge \bigwedge^n_{k=1} i dx_k\wedge d\Bar{x}_k\big).  
\end{aligned}
\end{equation}

\begin{lemma}\label{+ det}
  Suppose $u$ is a $C^2$ function on $D\times X$ and $\omega +i\partial\Bar{\partial}u(z,\cdot)>0$ on $X$ for all $z\in D$.  Then $u$ is $\omega$-subharmonic on graphs if and only if $$(i\partial \bar{\partial}u+\pi^*\omega)^{n+1}\wedge(i\sum_{j=1}^m dz_j\wedge d\bar{z}_j)^{m-1} \geq 0.$$
\end{lemma}

\begin{proof}
Let $\psi$ be a local potential of $\omega$ and denote the complex Hessian of $u+\psi$ with respect to $z_j$ and $x$ by $(u+\psi)_j$, as in the matrix (\ref{matrix}). Due to (\ref{invariant}), we will focus on $\sum^m_{j=1}\det (u+\psi)_j$.

 Let $f$ be a holomorphic function from an open subset of $D$ to $X$, then in a coordinate system on $X$ 
\begin{align*}
&\Delta(\psi(f(z))+u(z,f(z)))=\\
    &\sum_{i ,\lambda, \mu}\psi_{ \mu\Bar{\lambda}}\frac{\partial f^{\mu}}{\partial z_i}\frac{\partial \Bar{f^{\lambda}}}{\partial \Bar{z_i}}+\sum_i u_{i\Bar{i}}+\sum_{i ,\lambda}u_{i \Bar{\lambda}}\frac{\partial \Bar{f^{\lambda}}}{\partial \Bar{z_i}}+\sum_{i,\mu}u_{\Bar{i}\mu}\frac{\partial f^{\mu}}{\partial z_i}+\sum_{i, \lambda, \mu}u_{ \mu\Bar{\lambda}}\frac{\partial f^{\mu}}{\partial z_i}\frac{\partial \Bar{f^{\lambda}}}{\partial \Bar{z_i}}. 
\end{align*} 
If we denote the matrix $(\psi_{\mu\Bar{\lambda}}+u_{\mu\Bar{\lambda}})$ by $A$ and the column vector $(u_{i\bar{\lambda}})$ by $B_i$, then the above is the same as 
\begin{equation}
    \sum_i \big( \langle A\frac{\partial f}{\partial z_i}, \frac{\partial f}{\partial z_i}\rangle+\langle B_i,\frac{\partial f}{\partial z_i}  \rangle+\overline{\langle B_i,\frac{\partial f}{\partial z_i}  \rangle}+u_{i\bar{i}} \big),
\end{equation}
where the angled inner product is the usual Euclidean inner product and $\partial f/\partial z_i$ is the column vector $(\partial f^{\mu}/\partial z_i)$. The matrix form can be further written as 
\begin{equation}\label{15}
    \sum_i \big(\|\sqrt{A}\frac{\partial f}{\partial z_i}+\sqrt{A}^{-1}B_i\|^2-\|\sqrt{A}^{-1}B_i\|^2+u_{i\Bar{i}}\big).
\end{equation}
 Notice that 
 \begin{equation}\label{16}
 \begin{aligned}
 \sum_i (-\|\sqrt{A}^{-1}B_i\|^2+u_{i\Bar{i}})&=\sum_i(u_{i\Bar{i}}-\langle A^{-1}B_i,B_i\rangle)=\sum_i(u_{i\Bar{i}}-\sum_{\lambda, \mu}u_{i\bar{\lambda}}(\psi+u)^{\bar{\lambda} \mu}u_{\Bar{i}\mu})\\
 &=\sum_i\frac{\det(u+\psi)_i}{\det (\psi_{\mu\Bar{\lambda}}+u_{\mu\Bar{\lambda}})}, 
 \end{aligned}
 \end{equation}
where the last equality can be deduced from Schur's formula for determinants of block matrices as follows (see also \cite{Sem92} and \cite{Bern09} for a different computation). We examine the complex Hessian of $u+\psi$ 
\[(u+\psi)_j=
\left (
\begin{array}{cccc}
(u+\psi)_{z_j\bar{z}_j}
& (u+\psi)_{z_j\bar{x}_1} & \cdots &  (u+\psi)_{z_j\Bar{x}_n}\\
(u+\psi)_{x_1\bar{z}_j} & (u+\psi)_{x_1\bar{x}_1} & \cdots & (u+\psi)_{x_1\bar{x}_n}\\
 \vdots & \vdots & \ddots & \vdots \\
(u+\psi)_{x_n\bar{z}_j}& (u+\psi)_{x_n\bar{x}_1} &\cdots & (u+\psi)_{x_n\bar{x}_n} 
\end{array}
\right ), 
\]
and the Schur complement of the trailing $n\times n$ minor $((u+\psi)_{\mu\Bar{\lambda}})$ is precisely $u_{j\Bar{j}}-\sum_{\lambda, \mu}u_{j\Bar{\lambda}}(u+\psi)^{\Bar{\lambda} \mu}u_{\Bar{j}\mu}$, which is also equal to $\det (u+\psi)_j/\det ((u+\psi)_{\mu\Bar{\lambda}})$ by Schur's formula, for example see \cite{HZ05}. 

$u$ is $\omega$-subharmonic on graphs if and only if (\ref{15}) is nonnegative for any holomorphic maps $f$, and it is equivalent to the last term in (\ref{16}) being nonnegative. The lemma follows since the matrix $(\psi_{\mu\Bar{\lambda}}+u_{\mu\Bar{\lambda}})$ is positive and the equality (\ref{invariant}). 
\end{proof}
From (\ref{invariant}) and (\ref{16}), the function $K$ in the proof of Theorem \ref{thm bo bundle} has the following invariant expression 
\begin{equation*}
    K=\frac{m!n!}{(m-1)!(n+1)!}\frac{(\pi^*\omega+i\partial \bar{\partial}u)^{n+1}\wedge(i\sum_{j=1}^m dz_j\wedge d\bar{z}_j)^{m-1}}{(\omega+i\partial \bar{\partial}u)^n\wedge (i\sum_{j=1}^m dz_j\wedge d\bar{z}_j)^{m} }, 
\end{equation*}
and one can see $K\geq 0$ if $u$ is $\omega$-subharmonic on graphs.

\begin{proof}[Proof of Theorem \ref{thm har}]
 By the equality (\ref{invariant}), the equation $$(i\partial \bar{\partial}V+\pi^*\omega)^{n+1}\wedge(i\sum_{j=1}^m dz_j\wedge d\bar{z}_j)^{m-1}=0$$ is equivalent to $\sum_j\det(\psi+V)_j=0$, so we will prove the latter equation.
  
 By Corollary \ref{cor stable}, $V$ is $\omega$-subharmonic on graphs, and hence $V(z,x)$ is $\omega$-psh on $X$ by Lemma \ref{lemma psh in X}. Take a coordinate chart $\Omega$ of $X$, then for $\varepsilon>0$ and $x\in \Omega$,  the function $V(z,x)+\varepsilon|x|^2$ satisfies the assumption of Lemma \ref{+ det}, so $\sum_i\det(\psi+V+\varepsilon|x|^2)_i\geq 0$ and $\sum_i\det(\psi+V)_i\geq 0$. 
 
 Suppose $\sum_i\det(\psi+V)_i$ is positive at a point $p$ in $D\times X$. We may assume $\det(\psi+V)_1$ is positive at $p$, so it is positive in a neighborhood $B$ of $p$ in $D\times X$. For small $\varepsilon>0$, $\det (\psi+V+\varepsilon|x|^2)_1>0$ on $B$, then by Sylvester's criterion for positive matrices or a property of Schur complement for positive matrices (See \cite[Theorem 1.12]{HZ05}) we deduce that the matrix $(\psi+V+\varepsilon|x|^2)_1$ is positive on $B$, so the matrix $(\psi+V)_1$ is semi-positive on $B$, but since $\det(\psi+V)_1$ is positive on $B$, $(\psi+V)_1$ is actually positive on $B$; in particular, the $n\times n$ trailing minor $(\psi_{\lambda\Bar{\mu}}+V_{\lambda\Bar{\mu}})$ is positive on $B$. 
 
 So if we pick a suitably small smooth cutoff function $\rho$ supported in $B$, then the function $V+\rho$ satisfies the assumption of Lemma \ref{+ det} on $B$, and hence $V+\rho$ is $\omega$-subharmonic on graphs  and is in $G_v$, which contradicts $V=\sup G_v$. Therefore, $\sum_j\det(\psi+V)_j= 0$.
\end{proof}

As in the introduction, $\theta$ on $\mathcal{H}_\omega$ is
\begin{equation}
  \theta(\xi_1,\xi_2,\xi_3):=g_M(\{\xi_1,\xi_2\}_{\omega_\phi},\xi_3)=\int_X\{\xi_1,\xi_2\}_{\omega_\phi}\xi_3 \omega_\phi^n,  
\end{equation}
where $\phi\in \mathcal H_\omega$ and $ \xi_1,\xi_2,\xi_3 \in T_\phi\mathcal H_\omega$. According to $\{\xi_1,\xi_2\}_{\omega_\phi} \omega_\phi^n=nd\xi_1\wedge d\xi_2 \wedge \omega_\phi^{n-1}$ and an integration by parts, we deduce $\int_X\{\xi_1,\xi_2\}_{\omega_\phi}\xi_3 \omega_\phi^n=\int_X \xi_1 \{ \xi_2,\xi_3\}_{\omega_\phi} \omega_\phi^n$, and therefore $\theta$ is indeed skew-symmetric and a three form. The rest of this section is devoted to showing that the three form $\theta$ is $d$-closed, and the derivation of Euler--Lagrange equation of $\mathscr{E}$. 

\begin{lemma}\label{d formula}
  Let $\beta$ be a $k$-form on $\mathcal H_\omega$, and let $\xi_0,...\xi_k$ be vector fields on $\mathcal H_\omega$, which are constant in the canonical trivialization $T\mathcal{H}_\omega \approx \mathcal H_\omega \times C^\infty (X)$. Then 
  \begin{align}
      d\beta (\xi_0,...\xi_k)=\sum^k_{j=0} (-1)^j \xi_j\beta(\xi_0,...,\hat{\xi_j},...,\xi_k),
  \end{align}
  where $\hat{\xi_j}$ means $\xi_j$ is to be omitted. (This formula is ture if $\mathcal H_\omega\subset C^\infty(X)$ is replaced by an open subset of a Fr\'echet space.)
\end{lemma}
\begin{proof}
This is a well-known formula except there should be terms involving Lie brackets on the right hand side, but since $\xi_j$ are constant vector fields, their Lie brackets are zero. 
\end{proof}

\begin{lemma}\label{theta is closed}
  The three-form $\theta$ is $d$-closed.
\end{lemma}
\begin{proof}
This is similar to the derivation of the Aubin--Yau functional and the Mabuchi energy (see e.g. \cite[Sectioin 4]{Blo13}). Consider four vector fields $ \xi_1,\xi_2,\xi_3, \xi_4$ on $\mathcal H_\omega$, which are constant in the canonical trivialization $T\mathcal{H}_\omega \approx \mathcal H_\omega \times C^\infty (X) $. By Lemma \ref{d formula} 
 \begin{equation}\label{4 form}
     d\theta(\xi_1,\xi_2,\xi_3,\xi_4)=\xi_1\theta(\xi_2,\xi_3,\xi_4)-\xi_2\theta(\xi_1,\xi_3,\xi_4)+\xi_3\theta(\xi_1,\xi_2,\xi_4)-\xi_4\theta(\xi_1,\xi_2,\xi_3).
 \end{equation}
Using $\{\xi_3,\xi_4\}_{\omega_\phi} \omega_\phi^n=nd\xi_3\wedge d\xi_4 \wedge \omega_\phi^{n-1}$ and $  \frac{d}{dt}\Bigr|_{t=0} \omega_{\phi+t\xi_1}^{n-1} =(n-1)i\partial\Bar{\partial}\xi_1\wedge\omega_\phi^{n-2}$,  
 \begin{align*}
     \xi_1\theta(\xi_2,\xi_3,\xi_4)&=\xi_1\theta(\xi_3,\xi_4,\xi_2)=d (\theta(\xi_3,\xi_4,\xi_2))(\xi_1)=\frac{d}{dt}\Bigr|_{t=0} \theta(\xi_3,\xi_4,\xi_2)(\phi+t\xi_1)\\
     &= \frac{d}{dt}\Bigr|_{t=0} \int_X\{\xi_3,\xi_4\}_{\omega_{\phi+t\xi_1}}\xi_2 \omega_{\phi+t\xi_1}^n\\
     &=\frac{d}{dt}\Bigr|_{t=0} \int_X\xi_2 nd\xi_3\wedge d\xi_4 \wedge \omega_{\phi+t\xi_1}^{n-1}\\
     &=\int_X\xi_2 nd\xi_3\wedge d\xi_4 \wedge (n-1)i\partial\Bar{\partial}\xi_1\wedge\omega_\phi^{n-2}\\
     &=\int_X\xi_1 nd\xi_3\wedge d\xi_4 \wedge (n-1)i\partial\Bar{\partial}\xi_2\wedge\omega_\phi^{n-2}=\xi_2\theta(\xi_1,\xi_3,\xi_4),
     \end{align*}
 where the second to last equality is due to integration by parts. Because of this symmetry in index, (\ref{4 form}) is 0 and therefore $d\theta=0$. 
\end{proof}

Since $\theta$ is $d$-closed, there exists a two-form $\alpha$ on $\mathcal H_\omega$ such that $d\alpha=\theta$. For a map $\Phi: D\to \mathcal{H}_\omega$, the derivative $\Phi_{{z_j}}=1/2(\Phi_{\operatorname{Re}z_j}-i\Phi_{\operatorname{Im}z_j})$ is a section of $\mathbb{C}\otimes T\mathcal H_\omega$ along $\Phi$, and $\alpha(\Phi_{\bar{z}_j},\Phi_{{z_j}})$ is a function on $D$. We define $\mathscr{E}$ by $$\mathscr{E}(\Phi):=E(\Phi)+4 i\sum_j\int_D\alpha (\Phi_{\bar{z}_j},\Phi_{{z_j}})dV=\int_D |\Phi_*|^2 dV+4 i\sum_j\int_D\alpha (\Phi_{\bar{z}_j},\Phi_{{z_j}})dV,$$ with $dV$ the Euclidean volume form on $D$. 

\begin{lemma}\label{EL for WZW}
  The Euler--Lagrange equation of $\mathscr{E}$ is
\begin{equation}
  \sum^m_{j=1}|\nabla\Phi_{z_j}|^2-2\Phi_{z_j\bar{z}_j}+i\{\Phi_{\bar{z}_j},\Phi_{{z_j}}\}_{\omega_\Phi}=0,  
\end{equation}
where $\nabla\Phi_{z_j}$ is the gradient of $\Phi_{z_j}$ with respect to the metric $\omega_\Phi$.
\end{lemma}
\begin{proof}
Let $\Psi$ be a smooth map from $D$ to $C^\infty(X)$ with compact support. The variational equation is 
\begin{equation}\label{whole var}
    0=  \frac{d}{d t}\Bigr|_{t=0}\Bigr( \int_D |(\Phi+t\Psi)_*|^2 dV +   4i\sum_j  \int_D\alpha ((\Phi+t\Psi)_{\bar{z}_j},(\Phi+t\Psi)_{{z_j}}) dV \Bigr ).
\end{equation}
An extension of the computation in \cite[Section 2]{Don99} shows that the first term in (\ref{whole var}) is equal to  
\begin{equation}\label{Don}
    \frac{d}{d t}\Bigr|_{t=0} \int_D |(\Phi+t\Psi)_*|^2 dV=\int_D\int_X 4(\sum_j|\nabla \Phi_{z_j}|^2-2\sum_j \Phi_{z_j\bar{z}_j})\Psi \omega^n_\Phi dV. 
\end{equation}
So the remaining task is to compute the second term in (\ref{whole var}). 

We denote $C^\infty(X,\mathbb{C})$ by $C^\infty_{\mathbb{C}}(X)$. Introduce $A:\mathcal H_\omega\times C^\infty_{\mathbb{C}}(X)\times C^\infty_{\mathbb{C}}(X)\to \mathbb C$ as follows. If $(u,\xi),(u,\eta)\in \mathcal H_\omega\times C^\infty_{\mathbb C}(X)\approx \mathbb{C}\otimes T\mathcal{H}_\omega$, then 
$A(u,\xi,\eta):=\alpha((u,\xi),(u,\eta))$. Therefore, for fixed small $t\in \mathbb{R}$,  $\alpha ((\Phi+t\Psi)_{\bar{z}_j},(\Phi+t\Psi)_{{z_j}})=A(\Phi+t\Psi,(\Phi+t\Psi)_{\bar{z}_j}, (\Phi+t\Psi)_{{z_j}} ):D\to \mathbb{C}$. By chain rule, 
\begin{equation}\label{dt}
\begin{aligned}
&\frac{d}{d t}\Bigr|_{t=0}    A(\Phi+t\Psi,(\Phi+t\Psi)_{\bar{z}_j}, (\Phi+t\Psi)_{{z_j}} )\\
=&d_1A(\Phi,\Phi_{\bar{z}_j},\Phi_{z_j})(\Psi)+d_2A(\Phi,\Phi_{\bar{z}_j},\Phi_{z_j})(\Psi_{\bar{z}_j})+d_3A(\Phi,\Phi_{\bar{z}_j},\Phi_{z_j})(\Psi_{z_j}),
\end{aligned}
\end{equation}
where $d_1A, d_2A$, and $d_3A$ are partial differentials of $A$. Since $A$ is linear in the second and the third variables, $d_2A(\Phi,\Phi_{\bar{z}_j},\Phi_{z_j})(\Psi_{\bar{z}_j})=A(\Phi,\Psi_{\bar{z}_j},\Phi_{z_j})$ and $d_3A(\Phi,\Phi_{\bar{z}_j},\Phi_{z_j})(\Psi_{z_j})=A(\Phi,\Phi_{\bar{z}_j},\Psi_{z_j})$. Hence (\ref{dt}) becomes 
\begin{equation}\label{25}
    d_1A(\Phi,\Phi_{\bar{z}_j},\Phi_{z_j})(\Psi)+A(\Phi,\Psi_{\bar{z}_j},\Phi_{z_j})+A(\Phi,\Phi_{\bar{z}_j},\Psi_{z_j}).
\end{equation}
By similar computations, 
\begin{equation}
    \begin{aligned}
    \frac{\partial}{\partial \bar{z}_j}A(\Phi,\Psi, \Phi_{z_j})=& d_1A(\Phi,\Psi, \Phi_{z_j})(\Phi_{\bar{z}_j})+A(\Phi,\Psi_{\bar{z}_j}, \Phi_{z_j})+ A(\Phi,\Psi, \Phi_{z_j\bar{z}_j}), \textup{ and also}\\
\frac{\partial}{\partial z_j}A(\Phi,\Phi_{\bar{z}_j}, \Psi)=&d_1A(\Phi,\Phi_{\bar{z}_j},\Psi)(\Phi_{z_j})+A(\Phi,\Phi_{\bar{z}_jz_j},\Psi)+A(\Phi,\Phi_{\bar{z}_j},\Psi_{z_j}).  
\end{aligned}
\end{equation}
So integration by parts gives 
\begin{equation}\label{int by parts}
    \begin{aligned}
    \int_D A(\Phi,\Psi_{\bar{z}_j}, \Phi_{z_j})dV=-\int_D \Big(  d_1A(\Phi,\Psi, \Phi_{z_j})(\Phi_{\bar{z}_j})+A(\Phi,\Psi, \Phi_{z_j\bar{z}_j})\Big) dV,\\
    \int_D A(\Phi,\Phi_{\bar{z}_j},\Psi_{z_j})dV=-\int_D \Big( d_1A(\Phi,\Phi_{\bar{z}_j},\Psi)(\Phi_{z_j})+A(\Phi,\Phi_{\bar{z}_jz_j},\Psi)     \Big)dV. 
    \end{aligned}
\end{equation}
Combining (\ref{25}) and (\ref{int by parts})
\begin{equation}\label{28}
\begin{aligned}
       &\frac{d}{d t}\Bigr|_{t=0}  \int_D\alpha((\Phi+t\Psi)_{\bar{z}_j},(\Phi+t\Psi)_{{z_j}})dV\\
       =&\int_D d_1A(\Phi,\Phi_{\bar{z}_j},\Phi_{z_j})(\Psi)-d_1A(\Phi,\Psi, \Phi_{z_j})(\Phi_{\bar{z}_j})-d_1A(\Phi,\Phi_{\bar{z}_j},\Psi)(\Phi_{z_j})dV.
       \end{aligned}
\end{equation}

For a fixed point $z_0\in D$, $\Psi(z_0),\Phi_{\bar{z}_j}(z_0),$ and $\Phi_{z_j}(z_0)$ define three constant vector fields on $\mathcal H_\omega$, and we denote them by $\xi_1,\xi_2$, and $\xi_3$ respectively. By Lemma \ref{d formula}, $d\alpha(\xi_1,\xi_2,\xi_3)=\xi_1\alpha(\xi_2,\xi_3)-\xi_2\alpha(\xi_1,\xi_3)+\xi_3\alpha(\xi_1,\xi_2)$. Meanwhile, for constant vector fields $\xi_1, \xi_2, \xi_3$, the function $\xi_1\alpha(\xi_2,\xi_3)$ evaluated at $u\in \mathcal{H}_\omega$ is $d_1A(u,\xi_2,\xi_3)(\xi_1)$. So at $\Phi(z_0)\in \mathcal H_\omega$,  
\begin{equation}
\begin{aligned}
 d\alpha(\xi_1,\xi_2,\xi_3) =&d_1A(\Phi(z_0),\xi_2,\xi_3)(\xi_1)-d_1A(\Phi(z_0),\xi_1,\xi_3 )(\xi_2)+d_1A(\Phi(z_0),\xi_1,\xi_2)(\xi_3)\\
 =&d_1A(\Phi(z_0),\xi_2,\xi_3)(\xi_1)-d_1A(\Phi(z_0),\xi_1,\xi_3 )(\xi_2)-d_1A(\Phi(z_0),\xi_2,\xi_1)(\xi_3).
 \end{aligned}
\end{equation}
Hence (\ref{28}) becomes 
\begin{equation}\label{30}
\int_D d\alpha(\Psi,\Phi_{\bar{z}_j},\Phi_{z_j}) dV=\int_D \theta(\Psi,\Phi_{\bar{z}_j},\Phi_{z_j})dV=\int_D \int_X\{\Phi_{\bar{z}_j},\Phi_{{z_j}}\}_{\omega_\Phi}\Psi \omega_\Phi^n  dV.
\end{equation}

Finally, with (\ref{Don}) and (\ref{30}), the variational equation (\ref{whole var}) becomes 
\begin{equation}
\begin{aligned}
0=\int_D\int_X \Big( 4(\sum_j|\nabla \Phi_{z_j}|^2-2\sum_j \Phi_{z_j\bar{z}_j})  +4i\sum_j  \{\Phi_{\bar{z}_j},\Phi_{{z_j}}\}_{\omega_\Phi}\Big) \Psi \omega_\Phi^n dV,
\end{aligned}
\end{equation}
and we obtain the Euler--Lagrange equation $$\sum_j|\nabla \Phi_{z_j}|^2-2\sum_j \Phi_{z_j\bar{z}_j} + i\sum_j  \{\Phi_{\bar{z}_j},\Phi_{{z_j}}\}_{\omega_\Phi}=0.$$
 \end{proof}

\section{Lemma \ref{approx}}\label{section approx}
This section is mainly devoted to the proof of Lemma \ref{approx}, and we will follow closely the ideas in \cite{BK07}. The first two lemmas, concerning smooth approximation of continuous $\eta$-subharmonic functions, are based on Demailly's exposition \cite[Chapter I, Section 5E]{Dem12} of Richberg's paper \cite{Ric68}. 

Let $\theta\in C^\infty(\mathbb{R},\mathbb{R})$ be a nonnegative function having support in $[-1,1]$ with $\int_{\mathbb{R}}\theta(h)dh=1$ and $\int_{\mathbb{R}}h\theta(h)dh=0$. For arbitrary $\xi=(\xi_1,...,\xi_p)\in(0,\infty)^p$, the regularized maximal function is 
\begin{equation*}
    M_\xi(t_1,...,t_p):=\int_{\mathbb{R}^n}\max\{t_1+h_1,...,t_p+h_p\}\prod^n_{j=1}\theta(\frac{h_j}{\xi_j})\frac{dh_1}{\xi_1}...\frac{dh_p}{\xi_p}.
\end{equation*} 

\begin{lemma}\label{M}
  Fix a closed smooth positive $(1,1)$-form $\eta$ on $X$. Let $\Omega_\alpha \subset\subset D\times X$ be a locally finite open cover of $D\times X$, $c$ be a real number, and $u_\alpha\in C^\infty(\overline{\Omega}_\alpha)$ such that $u_\alpha(z,x)+c|z|^2$ is $\eta$-subharmonic on graphs. Assume that there exists a family $\{\xi_\alpha\}$ of positive numbers such that, for all $\beta$ and $(z,x)\in \partial \Omega_\beta$, $$u_\beta(z,x)+\xi_\beta \leq \max\{u_\alpha(z,x)-\xi_\alpha: \alpha \textup{ such that } (z,x)\in \Omega_\alpha  \}.$$ 
  
  Define a function $\tilde{u}$ on $D\times X$ as follows. Given $(z,x)\in D\times X$, let $A=\{\alpha:(z,x)\in \Omega_\alpha\}$, $\xi_A=(\xi_\alpha)_{\alpha\in A}$, and $u_A(z,x)=\{u_\alpha(z,x):\alpha\in A\}$, and  
  \begin{equation*}
  \Tilde{u}(z,x):=M_{\xi_A}(u_A(z,x)). 
  \end{equation*}
  Then $\tilde{u}$ is in $C^\infty(D\times X)$ and $\Tilde{u}(z,x)+c|z|^2$ is $\eta$-subharmonic on graphs.
\end{lemma}
\begin{proof}
 As in the proof of \cite[Chapter I, Lemma 5.17 and Corollary 5.19]{Dem12}, one can deduce that for a fixed point in $D \times X$,  there exist a neighborhood $V$ and a finite set $I$ of indices $\alpha$ such that $V\subset \bigcap_{\alpha\in I}\Omega_\alpha$ on which $\Tilde{u}=M_{\xi_I}(u_I)$. As a result, by \cite[Lemma 5.18 (a)]{Dem12}, $\Tilde{u}$ is smooth on $D\times X$. Now for a holomorphic map $f$ from an open subset of $D$ to $X$, 
 \begin{align*}
   \Tilde{u}(z,f(z))+c|z|^2+\psi(f(z))&=c|z|^2+\psi(f(z))+M_{\xi_I}(u_I(z,f(z)))\\
   &=M_{\xi_I}\big(c|z|^2+\psi(f(z))+u_I(z,f(z))\big), 
 \end{align*}
where $\eta=i\partial\bar{\partial}\psi$, and we use \cite[Lemma 5.18 (d)]{Dem12} in the last equality; furthermore, since $c|z|^2+\psi(f(z))+u_\alpha(z,f(z))$ is subharmonic by assumption, so is $M_{\xi_I}(c|z|^2+\psi(f(z))+u_I(z,f(z)))$ by \cite[Lemma 5.18 (a)]{Dem12}, and therefore  $\Tilde{u}+c|z|^2$ is $\eta$-subharmonic on graphs.
\end{proof}

 We introduce here some notations that will be used later. Let $\rho_{1},\rho_{2}$ be kernels (i.e. nonnegative radial smooth functions with support in the unit ball and having integral one) in $\mathbb{C}^m$ and $\mathbb{C}^n$ respectively. For $\varepsilon>0$, $\rho_{1,\varepsilon}(\cdot):=\varepsilon^{-m} \rho_1(\cdot /\varepsilon)$, and $\rho_{2,\varepsilon}$ is similarly defined.
 
 The proof of the following lemma is very similar to that of \cite[Chapter 1, Theorem 5.21]{Dem12}.
 
\begin{lemma}\label{Richberg}
 Let $u\in C(D\times X)$ be $\eta$-subharmonic on graphs. For any number $\lambda>0$, there exists $\Tilde{u}\in C^\infty(D\times X)$ such that $u\leq \Tilde{u}\leq u+M\lambda$ where $M$ depends only on the diameter of $D$, and $\Tilde{u}$ is $(1+\lambda)\eta$-subharmonic on graphs.
\end{lemma}
\begin{proof}
 Let $\{\Omega_\alpha\}$ be a locally finite open cover of $D\times X$ by relatively compact open balls contained in coordinate patches of $D\times X$. Choose concentric balls $\Omega_\alpha''\subset \Omega_\alpha'\subset \Omega_\alpha$ of radii $r_\alpha''<r_\alpha'<r_\alpha$ and center $(c_\alpha,0)$ in the given coordinates $(z,x)$ near $\overline{\Omega}_\alpha$, such that $\Omega_\alpha''$ still cover $D\times X$, and $\eta$ has a local potential $\psi_\alpha$ in a neighborhood of $\overline{\Omega}_\alpha$. For small $\varepsilon_\alpha>0$ and $\delta_\alpha>0$, we set $$u_\alpha(z,x)=\bigr((u+\psi_\alpha)*\rho_{\varepsilon_\alpha}\bigr)(z,x)-\psi_\alpha(x)+\delta_\alpha(r_\alpha'^2-|z-c_\alpha|^2-|x|^2) \textup{ on } \overline{\Omega}_\alpha,$$ 
 where $*\rho_{\varepsilon_\alpha}$ is the convolution with $\rho_{\varepsilon_\alpha}:=\rho_{1,\varepsilon_\alpha}\rho_{2,\varepsilon_\alpha}$. Since $\psi_\alpha(x)+u(z,x)$ is subharmonic in $z$ and psh in $x$ by Lemma \ref{lemma psh in X}, the functions $(\psi_\alpha+u)*\rho_{\varepsilon_\alpha}$ decrease to $\psi_\alpha +u$ as $\varepsilon_\alpha$ goes to 0, locally uniformly because $u$ is continuous. For $\varepsilon_\alpha$ and $\delta_\alpha$ small enough, we have $u_\alpha\leq u+\lambda/2$ on $\overline{\Omega}_\alpha$; moreover, for any holomorphic map $f$ from an open subset of $D$ to $X$, 
 \begin{align*}
     \Delta\bigr(u_\alpha(z,f(z))+\psi_\alpha(f(z))\bigr)&=\Delta \bigr( (u+\psi_\alpha)*\rho_{\varepsilon_\alpha}\bigr) (z,f(z))-\delta_\alpha\Delta(|z-c_\alpha|^2+|f(z)|^2)\\
     &\geq -\delta_\alpha\Delta(|z-c_\alpha|^2+|f(z)|^2)\\
     &\geq -\lambda \Delta |z|^2-\lambda\Delta\psi_\alpha(f(z)),
 \end{align*}
 where the first inequality is due to the fact $(u+\psi_\alpha)*\rho_{\varepsilon_\alpha}$ is subharmonic on holomorphic graphs, which can be verified easily using $(u+\psi_\alpha)$ is subharmonic on holomorphic graphs. So $u_\alpha(z,x)+\lambda|z|^2$ is $(1+\lambda)\eta$-subharmonic on graphs. Set $$\xi_\alpha=\delta_\alpha\min\{r_\alpha'^2-r_\alpha''^2,(r_\alpha^2-r_\alpha'^2)/2\}.$$ Choose first $\delta_\alpha$ such that $\xi_\alpha<\lambda/2$, and then $\varepsilon_\alpha$ so small that $u\leq (u+\psi_\alpha)*\rho_{\varepsilon_\alpha}(z,x)-\psi_\alpha(x)<u+\xi_\alpha$ on $\overline{\Omega}_\alpha
 $. As $\delta_\alpha(r_\alpha'^2-|z-c_\alpha|^2-|x|^2)$ is $\leq -2\xi_\alpha$ on $\partial\Omega_\alpha$ and $>\xi_\alpha$ on ${\Omega}_\alpha''$, we have $u_\alpha<u-\xi_\alpha$ on $\partial\Omega_\alpha$ and $u_\alpha>u+\xi_\alpha$ on $\Omega_\alpha''$, so that the assumption in Lemma \ref{M} is satisfied, and the function 
 \begin{equation*}
     U(z,x):=M_{\xi_A}(u_A(z,x)), \textup{ for $A=\{\alpha$: $\Omega_\alpha \ni (z,x)$\}  ,}
 \end{equation*}
 is in $C^\infty(D\times X)$ and $U(z,x)+\lambda|z|^2$ is $(1+\lambda)\eta$-subharmonic on graphs. By \cite[Lemma 5.18 (b)]{Dem12}, $u\leq U\leq u+\lambda$. Then the function defined by $\Tilde{u}:=U+\lambda|z|^2$ is what we need.
\end{proof}

The following lemma is proved in the same way as Lemmas 4 and 5 in \cite{BK07}. The only issue is keeping track of uniformity.

\begin{lemma}\label{local uni}
Let $U,V$ be two open sets in $\mathbb{C}^n$ and $F$ a biholomorphic map from $U$ to $V$. Let $u$ be usc, bounded, and subharmonic on holomorphic graphs in $D\times U$. Define $u_{\delta_1,\delta_2}$ to be the convolution $$u_{\delta_1,\delta_2}(z,x)=\int\int u(z-a,x-b)\rho_{1,\delta_1}(a)\rho_{2,\delta_2}(b)da db,$$ where $\rho_{1,\delta_1},\rho_{2,\delta_2}$ are kernels in $\mathbb{C}^m$ and $\mathbb{C}^n$ respectively. On the other hand, define 
\begin{equation}\label{id F}
  u^F_{\delta_1,\delta_2}(z,x)=(u \circ (\Id\times F^{-1}))_{\delta_1,\delta_2}\circ (\Id\times F).  
\end{equation}
Then as $\delta_2\to 0$, $(u^F_{\delta_1,\delta_2}-u_{\delta_1,\delta_2})(z,x)$ goes to 0 locally uniformly in $z,x$, and $\delta_1$.
\end{lemma}
\begin{proof}
Define $$\hat{u}_{\delta_2}(z,x)=\max_{\{z\}\times \overline{B(x,\delta_2)}}u,$$ $$\Tilde{u}_{\delta_2}(z,x)=\fint_{\partial B(x,\delta_2)}u(z,b)d b, $$
$$u_{\delta_2}(z,x)=\int u(z,x-b)\rho_{2,\delta_2}(b) db,$$ where $\fint$ means the average.
Their counterparts under $\operatorname{Id}\times F^{-1}$ and $\operatorname{Id}\times F$ as in (\ref{id F}) are denoted by $\hat{u}^F_{\delta_2}(z,x), \Tilde{u}^F_{\delta_2}(z,x),$ and $ u^F_{\delta_2}(z,x)$ respectively. 

By Lemma \ref{lemma psh in X}, $u(z,\cdot)$ is psh in $U$, so $\hat{u}_{\delta_2}(z,x)$ is a convex function of $\log \delta_2$. Fixing $a \geq 1$ and $r>0$, choose $\delta_2$ so small that $0 \leq \frac{\log a}{\log \frac{r}{\delta_2}}\leq 1$, then by convexity $$0\leq \hat{u}_{a\delta_2}(z,x)-\hat{u}_{\delta_2}(z,x)\leq \frac{\log a}{\log \frac{r}{\delta_2}}(\hat{u}_{r}(z,x)-\hat{u}_{\delta_2}(z,x)).$$ Since $u$ is assumed to be bounded, it follows that for any $a>0$ (for the case $1>a>0$, use $1/a$ instead), $\hat{u}_{a\delta_2}(z,x)-\hat{u}_{\delta_2}(z,x)$ goes to 0 as $\delta_2\to 0$, locally uniformly in $z$ and $x$. Then following the same argument as in \cite[Lemma 4]{BK07}, we see $\hat{u}^F_{\delta_2}-\hat{u}_{\delta_2}$ goes to 0 locally uniformly in $z$ and $x$, as $\delta_2\to 0$.


Since $u(z,\cdot)$ is psh in $U$, $\Tilde{u}_{\delta_2}(z,x)$ is convex in $\log \delta_2$. By the argument \cite[Lemma 5]{BK07} and the fact that $u$ is bounded, we see both $\hat{u}_{\delta_2}-\Tilde{u}_{\delta_2}$ and $\Tilde{u}_{\delta_2}-u_{\delta_2}$ go to 0 locally uniformly in $z,x$, as $\delta_2\to 0$, and as a result, so does ${u}^F_{\delta_2}-{u}_{\delta_2}$. Since $(u^F_{\delta_1,\delta_2}-u_{\delta_1,\delta_2})$ is the convolution of (${u}^F_{\delta_2}-{u}_{\delta_2}$) in $z$, we see at once the conclusion of the lemma.
\end{proof}

\begin{proof}[Proof of Lemma \ref{approx}]
 Fix a finite number of charts $U_\alpha \supset\supset V_\alpha$ such that $V_\alpha$ covers $X$, and  $\eta$ has a local potential $\psi_\alpha$ in a neighborhood of $\overline{U_\alpha}$. For each $\alpha$, let $f_\alpha: U_\alpha \to \mathbb{C}^n$ be the coordinate map, we consider the convolution $((\psi_\alpha+u)\circ f^{-1}_\alpha)_{\delta_1,\delta_2}\circ f_\alpha$, which we simply denote by $(\psi_\alpha+u)_{\delta_1,\delta_2}$ on $D \times U_\alpha$. Because $u$ added by a constant still satisfies the same assumption in Lemma \ref{approx}, we will assume $u$ is so negative that $(\psi_\alpha+u)_{\delta_1,\delta_2}-\psi_\alpha<-a$ for some $a>0$ and all $\alpha$.  At the same time, we consider the convolution of $(\psi_\alpha+u)$ under $f_\beta$, namely $((\psi_\alpha+u)\circ f^{-1}_\beta)_{\delta_1,\delta_2}\circ f_\beta$, which can be written as \begin{equation}\label{convolution}
  ((\psi_\alpha+u)\circ f^{-1}_\alpha \circ F^{-1})_{\delta_1,\delta_2}\circ F\circ f_\alpha,  
\end{equation}
if $F^{-1}=f_\alpha\circ f^{-1}_\beta$. We denote (\ref{convolution}) by $(\psi_\alpha+u)^F_{\delta_1,\delta_2}$ (the notation is consistent with Lemma \ref{local uni} except we do not write out the identity map of $D$ here). By Lemma \ref{local uni} on $D\times (U_\alpha \cap U_\beta)$ 
\begin{equation}\label{4}
\begin{aligned}
    &(\psi_\alpha+u)_{\delta_1,\delta_2}-(\psi_\beta+u)_{\delta_1,\delta_2}=(\psi_\alpha+u)_{\delta_1,\delta_2}-(\psi_\alpha+u)^F_{\delta_1,\delta_2}+(\psi_\alpha+u-(\psi_\beta+u))^F_{\delta_1,\delta_2}\\
    &\to \psi_\alpha-\psi_\beta
\end{aligned}    
\end{equation}
locally uniformly in $z$ and $x$, as $\delta_2,\delta_1\to 0$. 

Let $\chi_\alpha$ be a smooth function in $U_\alpha$ that is 0 in $V_\alpha$ and $-1$ near $\partial U_\alpha$. We have $i\partial \bar{\partial}\chi_\alpha \geq -C\eta$ for some constant $C$. For $0<\varepsilon<1$, according to (\ref{4}) we can find $\delta_1,\delta_2$ so small that for any $\beta$ and for any $(z,x)\in \overline{D'}\times \partial U_\beta$, $$((\psi_\beta+u)_{\delta_1,\delta_2}-\psi_\beta+\frac{\varepsilon}{C}\chi_\beta)(z,x)< \max_{(z,x)\in \overline{D'}\times U_\alpha}((\psi_\alpha+u)_{\delta_1,\delta_2}-\psi_\alpha+\frac{\varepsilon}{C}\chi_\alpha)(z,x),$$ where the maximum is taken over all $\overline{D'}\times U_\alpha$ that contain $(z,x)$. Let $\delta=\min\{\delta_1,\delta_2\}$.
Then by \cite[Chapter I, Lemma 5.17]{Dem12}, the function   $${u}^{\varepsilon}_\delta(z,x):=\max_{(z,x)\in \overline{D'}\times U_\alpha}((\psi_\alpha+u)_{\delta,\delta}-\psi_\alpha+\frac{\varepsilon}{C}\chi_\alpha)(z,x) ,$$
is continuous on $\overline{D'}\times X$. Notice that ${u}^{\varepsilon}_\delta(z,x)<-a$ for any $0<\varepsilon<1$. Since $\psi_\alpha(x)+u(z,x)$ is subharmonic in $z$ and psh in $x$ by Lemma \ref{lemma psh in X}, the function $(\psi_\alpha+u)_{\delta,\delta}$ is decreasing to $\psi_\alpha +u$ as $\delta \to 0$, and hence $u^\varepsilon_\delta$ is decreasing to $u$ as $\delta \to 0$.

We already know that $\psi_\alpha+u$ is subharmonic on holomorphic graphs, and a straightforward verification shows so is $(\psi_\alpha+u)_{\delta,\delta}$. This fact together with $(\chi_{\alpha})_{\lambda\Bar{\mu}}\geq -C (\psi_\alpha)_{\lambda\Bar{\mu}}$ as martrices shows, for any holomorphic function $f$ from an open subset of $D'$ to $X$, $$\Delta ((\psi_\alpha+u)_{\delta,\delta}-\psi_\alpha+\frac{\varepsilon}{C}\chi_\alpha)(z,f(z))\geq (-1-\varepsilon)\Delta \psi_\alpha(f(z)), $$ so ${u}^{\varepsilon}_\delta$ is $(1+\varepsilon)\eta$-subharmonic on graphs. 

So far we have shown that given $1<p\in \mathbb{N}$, there exist $q_0\in \mathbb{N}$ such that, for $q > q_0$, the functions $u^{1/p}_{1/q}$ are in $C(\overline{D'}\times X)$, $(1+1/p)\eta$-subharmonic on graphs, and decrease to $u$ as $q \to \infty$. For simplicity, we will denote $u^{1/p}_{1/q}$ by $u^p_q$.
Let $M$ be the constant in Lemma \ref{Richberg}. We will construct inductively $u^k_{j_k}$ with $j_k>k^2$ and $\tilde{u}_k\in C^\infty(D'\times X)$ such that 
\begin{equation}
  u^k_{j_k}+1/j_k\leq \tilde{u}_k\leq u^{k}_{j_k}+1/j_k+M/j_k.
\end{equation}
Moreover $\tilde{u}_k$ is $(1+1/k)(1+1/j_k)\eta$-subharmonic on graphs, and $ u^{k}_{j_k}+1/j_k+M/j_k$ is less than both $u^{k-1}_{j_{k-1}}+1/j_{k-1}$ and $u^{2}_{j_{k-1}}+1/j_{k-1}$.

Suppose that this is true at $(k-1)$-th step. As $u^{k-1}_{j_{k-1}}+1/j_{k-1}$ and $u^{2}_{j_{k-1}}+1/j_{k-1}$ are both greater than $u$, we can find $j_k >\max\{j_{k-1},k^2\}$
such that $u^k_{j_k}+1/j_k+M/j_k$ is less than both $u^{k-1}_{j_{k-1}}+1/j_{k-1}$ and $u^{2}_{j_{k-1}}+1/j_{k-1}$ by continuity on the compact set $\overline{D'}\times X$. Applying Lemma \ref{Richberg} with $\lambda=1/j_k$, we find a function $\tilde{u}_k\in C^\infty(D'\times X)$ with $u^k_{j_k}+1/j_k\leq \tilde{u}_k\leq u^{k}_{j_k}+1/j_k+M/j_k$ and $\tilde{u}_k$ is  $(1+1/k)(1+1/j_k)\eta$-subharmonic on graphs. So the induction process is true at $k$-th step. (One can begin the induction process with $u^{2}_{j_2}+1/j_2$ with $j_2$ large enough that $u^{2}_{j_2}+1/j_2<0$).



One can see that $\tilde{u}_k$ is decreasing to $u$. Since $\tilde{u}_k<0$, $(1-1/k)\tilde{u}_k$ is still decreasing to $u$. The function $(1-1/k)\tilde{u}_k$ is $(1-1/k^2)(1+1/j_k)\eta$-subharmonic on graphs, which is also $(1-1/k^2j_k)\eta$-subharmonic on graphs because $j_k>k^2$. So $(1-1/k)\tilde{u}_k$ are the desired approximants.
\end{proof}


\section{A remark}\label{compare}

In this final section, we compare results in this paper to the work with Darvas \cite{DW19}, where we consider two other families closely related to $G_v$ and $G^k_v$. For $\pi:D\times X \to X$, $$F_v:=\{u:u\in \text{PSH}(D\times X,  \pi^*\omega)\textup{ and }\limsup_{D \ni z \to \zeta\in \partial D}u(z,x)  \leq v(\zeta,x)\},$$ 
\begin{align*}
  F_v^{k}:= \{&D \ni z \to U_z \in \mathcal N_k^* \textup{ is Griffiths negative and }\\ 
  &\limsup_{D\ni z \to \zeta \in \partial D} U^2_z(s) \leq H_k^*(v_\zeta)(s,s) \textup{ for any $s\in H^0(X,L^k)^*$ } \},  
\end{align*}
where a norm function $U_z$ is called Griffiths negative if $\log U_z(f(z))$ is psh for any holomorphic section $f:W\subset D \to H^0(X,L^k)^*$ . Denote the upper envelopes of $F_v$ and $F^k_v$ by $U$ and $U^{k}$ respectively; then one result in \cite{DW19} is that $FS_k((U^{k}_z)^*)$ converges to $U$ uniformly. The difference between $F^{k}_v$ and $G^{k}_v$ is obvious, we simply change plurisubharmonicity to subharmonicity. The relation of $F_v$ and $G_v$ can be seen as follows. Let $\psi$ be a local potential of $\omega$, then a  function $u\in\text{PSH}(D\times X,  \pi^*\omega)$ is equivalent to  $\psi(x)+u(z,x)$ being psh in $z$ and $x$ jointly, which is also equivalent to $\psi(f(z))+u(z,f(z))$ being psh for any holomorphic $f:U\subset D\to X$ (see Lemma \ref{psh} below); therefore we see the change from $F_v$ to $G_v$ is again plurisubharmonicity to subharmonicity. Also notice that when $\dim D=1$ Theorem \ref{thm:YM approx} and the result in \cite{DW19} are the same because $F_v=G_v$ and $F^{k}_v=G^{k}_v$.

The lemma below is to justify the transition from $F_v$ to $G_v$.
\begin{lemma}\label{psh}
  Let $\Omega_1$ and $\Omega_2$ be open sets in $\mathbb{C}^m$ and $\mathbb{C}^n$ respectively. If $u(z,\xi)$ is an usc function on $\Omega_1 \times \Omega_2$ such that $u(z,s(z))$ is psh for any holomorphic map $s$ from an open subset of $\Omega_1$ to $\Omega_2$, then $u$ is psh on $\Omega_1 \times \Omega_2$. 
 \end{lemma}  
\begin{proof}
 We want to show that $u$ is subharmonic on any complex line in $\Omega_1 \times \Omega_2$, and it suffices to consider the line $\mathbb{C}\ni \lambda \mapsto (\lambda z_0,\lambda \xi_0)$ where $(z_0,\xi_0)\in \Omega_1 \times \Omega_2$. In the case when $z_0$ and $\xi_0$ are both nonzero, we may assume $z_0=(1,0,...,0)$ and $\xi_0=(1,0,...,0)$. Let $G:\Omega_1 \to \mathbb{C}$ be the projection on first coordinate, and $F:\mathbb{C}\to \Omega_2$ be the injection to the first coordinate. By assumption, $u(z,F\circ G(z))$ is psh, so $\lambda \mapsto u(\lambda z_0, F\circ G(\lambda z_0))=u(\lambda z_0, \lambda \xi_0)$ is subharmonic.

If $\xi_0=0$, then $\lambda \mapsto u(\lambda z_0,0)$ is of course subharmonic. The final case is $z_0=0$ and $\xi_0=(1,0,...,0)$, and we need to show $\lambda \mapsto u(0,(\lambda,0,...,0))$ is subharmonic. Given $\varepsilon>0$ and $a\in\mathbb{C}$, the function $z\mapsto u(z, (z_1/\varepsilon+a,0,...,0))$ is psh, so $\lambda \mapsto u((\lambda-a) \varepsilon,0,...,0,(\lambda,0,...0))$ is subharmonic. Hence, $$u(0,(a,0,...,0))\leq \fint_{\partial B(a,r)}u((\lambda-a)\varepsilon,0,...,0,(\lambda,0,...,0))d \lambda,$$ for $r>0$. By Fatou's lemma and the fact $u$ is usc, $$\limsup_{\varepsilon\to 0}\int_{\partial B(a,r)}u((\lambda-a)\varepsilon,0,...,0,(\lambda,0,...,0))d \lambda \leq \int_{\partial B(a,r)}u(0,0,...,0,(\lambda,0,...,0))d \lambda.$$ As a result, $$u(0,(a,0,...,0))\leq \fint_{\partial B(a,r)}u(0,(\lambda,0,...,0))d \lambda.$$
\end{proof}


\textsc{Purdue University}

\texttt{\textbf{wu739@purdue.edu}}

\end{document}